\documentclass[a4, 12pt,leqno]{amsart}
\usepackage{lmodern}
\usepackage[english]{babel}
\usepackage{microtype}
\usepackage{mathtools}
\usepackage{mathabx}
\usepackage{bbm}
\usepackage{amsthm}
\usepackage{mathtools}
\usepackage{breakcites}

\usepackage{accents}
\usepackage{float}
\usepackage{physics}
\usepackage{a4wide}
\usepackage{algorithmic}
\usepackage[ruled]{algorithm2e}
\usepackage[active]{srcltx}
\usepackage[utf8]{inputenc}
\usepackage[T1]{fontenc}
\usepackage{lipsum}
\usepackage{upgreek}
\usepackage{listings}
\usepackage{wasysym}
\DeclareMathAlphabet{\mathdutchcal}{U}{dutchcal}{m}{n}
\SetMathAlphabet{\mathdutchcal}{bold}{U}{dutchcal}{b}{n}
\DeclareMathAlphabet{\mathdutchbcal}{U}{dutchcal}{b}{n}
\DeclareSymbolFont{myletters}{OML}{ztmcm}{m}{it}
\DeclareMathSymbol{\nicelambda}{\mathord}{myletters}{"15}
\usepackage{amsmath,amssymb,amscd,amsthm,graphicx,enumerate}
\usepackage{mathrsfs}
\usepackage{amssymb,amsthm,a4wide}
\usepackage{nicefrac}
\newcounter{example}

		\makeatletter
	\def\@xnamedef#1{\expandafter\protected@xdef\csname #1\endcsname}
	\def\no@harm{} 
	\def\ead@au#1{\protected@edef\@ead@au{#1}}
	\patchcmd\runningauthor@fmt{\global\edef}{\protected@xdef}{}{}
	\patchcmd\runningauthor@fmt{\global\edef}{\protected@xdef}{}{}
	\patchcmd\author@fmt{\edef}{\protected@edef}{}{}
	\patchcmd\add@xtok{\xdef}{\protected@xdef}{}{}
	\makeatother

\def\bysame{\leavevmode\hbox to3em{\hrulefill}\thinspace}
\usepackage{pifont}

\usepackage{color}
\usepackage[usenames,dvipsnames]{xcolor}
\usepackage{xcolor}
\definecolor{ultrablue}{rgb}{0.0,0.0, 1}
\definecolor{jigari}{rgb}{0.39,0.0, 0.0}
\usepackage{url}
\usepackage{hyperref}
\hypersetup{
	breaklinks=true,
	linktoc=page,
	colorlinks,
	linkcolor={jigari},
	citecolor={jigari},
	urlcolor={jigari}
}
\hypersetup{breaklinks=true}
\allowdisplaybreaks

\usepackage{tikz}
\usetikzlibrary{matrix}
\usepackage[all]{xy}
\usepackage{subfig}
\makeatletter 
\def\l@subsection{\@tocline{2}{0pt}{3pc}{6pc}{}}
\makeatother

\usepackage{enumerate}
\usepackage{soul}
\usepackage{graphicx}
\usepackage{tikz}
\usepackage{tikzpagenodes}
\usepackage{scrlayer-scrpage}
\usepackage{tikz}
\usetikzlibrary{calc,spy}
\usetikzlibrary{shapes.geometric, arrows}
\usepackage{pgfplots}
\usetikzlibrary{intersections, pgfplots.fillbetween}

\makeatletter
\@namedef{subjclassname@2020}{%
	\textup{2020} Mathematics Subject Classification}
\makeatother
\makeatletter 
\def\l@subsection{\@tocline{2}{0pt}{3pc}{6pc}{}}
 \makeatother
\usepackage[a4paper,bindingoffset=0in,%
left=1.2in,right=1.2in,top=1.5in,bottom=1.6 in,%
footskip=1mm]{geometry}
\theoremstyle{plain}
\newtheorem*{theorem*}{Theorem}
\newtheorem{lemma}{Lemma}[section]
\newtheorem{theorem}[lemma]{Theorem}

\theoremstyle{definition}
\newtheorem{definition}[lemma]{Definition}
\newtheorem{remark}[lemma]{Remark}

\theoremstyle{remark}

\numberwithin{equation}{section}

\newcommand{\R}{\mathbb R}

\DeclareMathOperator{\Ric}{Ric}
\DeclareMathOperator{\scal}{scal}
\newcommand{\m}{\mathdutchcal{m}}

\newcommand{\B}{\mathcal{B}}

\DeclareMathOperator{\dvol}{dvol}
\DeclareMathOperator{\Lip}{Lip}

\newcommand{\dist}{\mathsf{d}}
\newcommand{\dif}{\mathrm{d}}

\makeatletter
\newcommand{\tpmod}[1]{{\@displayfalse\pmod{#1}}}
\makeatother

\newcount\stylenum  \newdimen\styledim 
\def\varstyle#1{\mathchoice{\stylenum=0 #1}{\stylenum=1 #1}{\stylenum=2 #1}{\stylenum=3 #1}}
\def\mathaxis{\fontdimen22\ifcase\stylenum 
	\textfont\or\textfont\or\scriptfont\or\scriptscriptfont\fi2 }
\def\setstyledim{\styledim=\ifcase\stylenum .1em\or.1em\or.07em\or.05em\fi\relax}

\def\sqdot{\mathbin{\varstyle{\raise\mathaxis\hbox{\setstyledim
				\kern\styledim 
				\vrule width1.2\styledim height.6\styledim depth.6\styledim
				\kern\styledim}}}}

\newcommand{\restr}{\raisebox{-.1908ex}{$\big|$}}

\makeatletter
\newcommand{\tpitchfork}{%
	\vbox{
		\baselineskip\z@skip
		\lineskip-.52ex
		\lineskiplimit\maxdimen
		\m@th
		\ialign{##\crcr\hidewidth\smash{$-$}\hidewidth\crcr$\pitchfork$\crcr}
	}%
}
\makeatother

\newcommand{\ident}{\raisebox{0pt}{\scalebox{1.1}{$\mathbbm{1}$}}\hspace{-1pt}}

\usepackage{stackengine,scalerel}

\newcommand\overstar[1]{\ThisStyle{\ensurestackMath{%
			\setbox0=\hbox{$\SavedStyle#1$}%
			\stackengine{0pt}{\copy0}{\kern0\ht0\smash{\SavedStyle\star}}{O}{c}{F}{T}{S}}}}

\newcommand{\cost}{\mathrm{c}}

\newcommand{\Del}{{\Updelta}}

\newcommand{\Id}{\mathrm{Id}}

\DeclareMathOperator{\Haus}{\mathcal{H}}

\DeclareMathOperator{\vol}{vol}

\newcommand{\bal}{\mathdutchcal{b}}
\newcommand{\s}{\mathdutchcal{s}}

\newcommand{\Sph}{\mathbb{S}}

\newcommand{\Tang}{\mathbb{T}}
\newcommand{\eps}{\varepsilon}
\newcommand\smallO{
	\mathchoice
	{{\scriptstyle\mathcal{O}}}
	{{\scriptstyle\mathcal{O}}}
	{{\scriptscriptstyle\mathcal{O}}}
	{\scalebox{.7}{$\scriptscriptstyle\mathcal{O}$}}
}

\begin{document}
\title[\scriptsize {On weak formulations of (super) Ricci flows}]{\small On weak formulations of (super) Ricci flows} 

%

\author[\protect \scriptsize S. Lakzian]{Sajjad Lakzian$^\dagger$}
\address{\noindent -- Sajjad Lakzian \newline \noindent Department of Mathematical Sciences\newline Isfahan University of  Technology (IUT) \newline Isfahan 8415683111, Iran.}
\email{\href{mailto:slakzian@iut.ac.ir}{slakzian@iut.ac.ir}}
\thanks{$^\dagger$The author is supported by the INSF Grant No. 4030556, awarded by the ``On the Frontiers of Mathematical Sciences'' program.}
\maketitle
%
\begin{center}
\emph{In memory of Richard Hamilton, a brilliant visionary.} 
\end{center}
\begin{abstract}
\textsl{We present two characterizations of smooth compact Ricci flow solutions solely in terms of metrics and measures (one of them only works under positive scalar curvature along the flow); thus, provide weak formulations that are generalized to the singular setting in a straightforward manner. These formulations are achieved by weakly formulating super Ricci flows and imposing a saturation condition (solely in terms of metric and measure) to ensure the super Ricci flow inequality is an equality.
} 
\end{abstract}
\date{\today}
\section{Prelude}
\par Indisputably, the concept of the Ricci flow (see its backward form in~\eqref{EQN:bRF}) which can only be conceived in a genius and iconic mind like Richard Hamilton's, is one of the most useful and powerful tools in geometric analysis. Ricci flow began as a vision to tackle Poincare's conjecture (Hamilton~\cite{H1,H2,H3}) and sure enough it succeeded (Perelman~\cite{G1,G2,G3}); but by that time, it had morphed into something more than just a vision or a tool; it became a god given gift to Mathematics; a gift that keeps giving to this day. 
\par Unlike the mean curvature flow, the Ricci flow does not evolve in an ambient space; in other words, the Ricci flow is an intrinsic flow of the space. This significantly contributes  to the difficulty one faces in trying to weakly formulate it. To clarify, here, a weak formulation means one in terms of the metrics (distance functions) and measures. In other words, the quest is to formulate the Ricci flow in a way that it could be used for less regular spaces where classical differentiation in space variables is no longer an option.
\par The goal of the following notes is to provide two (very interrelated) such characterizations. We will present weak formulations for both Ricci flows and super Ricci flows on compact manifolds. The constructions utilized are sufficiently general to define analogs of the Ricci flow on metric measure spaces.
\par The main idea behind our notion of weak Ricci flow is quite simple: first, we weakly formulate super solutions to the Ricci flow (which are called super Ricci flows) and then, we impose a metric and measure condition that ensures the defining inequality of the super solution to the Ricci flow is saturated.
\par Under this exact name, super Ricci flows were first introduced in McCann-Topping~\cite{MT}. However, the coupled contraction that forms the basis of its definition can also be found in the concurrent paper Arnaudon-Coulibaly-Thalmaier~\cite{ACT}. Super Ricci flows are abundant, but the first interesting non-trivial and non-Riemannian example was presented in Lakzian-Munn~\cite{LM1}. Beyond Lakzian-Munn~\cite{LM1}, in the context of metric measure spaces, the notion of non-collapsed super Ricci flows was further developed and studied in Sturm~\cite{St1}. A generalized notion of super Ricci flows for time dependent metric-measure spaces with possible collapsed singularities (such as neckpinch singularities) and using a general convex cost function $\cost$ was presented in Lakzian-Munn~\cite{LM2}. The latter is based on the coupled contraction for dynamic diffusions and uses the concept of dynamic diffusions developed in Kopfer-Sturm~\cite{KS}. In this direction, also the recent work Bamler~\cite{B} provides a compactness theory for super Ricci flows which is of foundational importance in studying limit Ricci flows. We also note the interesting recent work Li~\cite{Li} in which the author has introduced and studied $(K,n,N)$-super Ricci flows in the setting of singular Bakry-\'Emery spaces and generalized some key results (such as the differential Harnack inequality) from the smooth Ricci flow theory to the said flows.  
\par We should mention that our approach to formulating weak (super) Ricci flow in these notes does not rely on the time-dependent theory developed in Sturm~\cite{St1}. Instead, is closely aligned with the one in Lakzian-Munn~\cite{LM1}. Other important and significantly different weak solutions of the Ricci flow can be found in Haslhofer-Naber~\cite{HN} and Kleiner-Lott~\cite{KL,KL2}. 

\subsection*{\bf Conventions/Notations:}
\renewcommand{\labelitemi}{\raisebox{2pt}{\scalebox{1.2}{$\centerdot$}}}
\begin{itemize}
	\item \textit{The flows we consider are all defined on compact intervals; we will not specify these intervals in most cases where they are not of essential importance.}
	\medskip
	\item \textit{Forward time parameters are mostly denoted by $t$ and $s$ and backward time by $\tau$ or $\sigma$. They are related via $\tau = T - t$ for some chosen reference time $T$. These will not be further specified.} 
	\medskip
	\item \textit{Throughout, the spaces we consider (manifolds, metric spaces and pseudo metric spaces) are assumed to be compact without further mentioning this.}
	\medskip
	\item \textit{The term ``static'' refers to the case where the generator is time-independent as opposed to ``dynamic'' which means time-dependent.}
	\medskip
	\item \textit{The results regarding heat flows are written in forward time, while the results concerning the conjugate heat flows are presented in backward time. $\overrightarrow{\mathcal{X}}$ and $\overleftarrow{\mathcal{X}}$ will denote forward and backward (resp.) time-dependent metric measure spaces.} 
	\medskip
	\item \textit{The obvious common sense conventions are in order. For example when $r=0$, the averaging operators are identity.} 
	\medskip
	\item \textit{Whenever we are discussing a one-parameter family of Riemannian metrics, the metrics are smooth and the dependence on the time parameter is at least $\mathcal{C}^1$.} 
	\medskip 
	\item \textit{Whenever an element is placed in front of an operator, we mean it is defined i.e. is within the domain of that operator without further mentioning the domain.}  
\end{itemize}

\addtocontents{toc}{\protect\setcounter{tocdepth}{-1}}
\section*{\small \bf  Acknowledgments}
\addtocontents{toc}{\protect\setcounter{tocdepth}{1}}

\vspace{-100pt}

\begin{minipage}[c][7cm][b]{0.87\textwidth}

	\renewcommand{\labelitemi}{\raisebox{1pt}{\scalebox{0.6}{\ding{169}}}}
\begin{itemize}
		\item \textit{The author would like to express his gratitude for the anonymous referee's valuable comments that immensely helped improving this manuscript, and for bringing to his attention the reference~Li~\cite{Li}};
		\medskip
		\item \textit{The author and this project are supported by the INSF Grant No. 4030556, awarded by the ``On the Frontiers of Mathematical Sciences'' program.} 
	\end{itemize}
\end{minipage}
\normalsize

\section{Preliminaries}
\subsection{Smooth super Ricci flows}
\par Super Ricci flows first appeared in McCann-Topping~\cite{MT} and were tied to the contraction properties in optimal transportation. The equation of Hamilton's Ricci flow -- parameterized in backward time (or backward Ricci flow) -- is
\begin{equation}
	\label{EQN:bRF}
	\frac{\partial g}{\partial \tau} = 2\Ric(g(\tau)). 
\end{equation}
\par A subsolution of (\ref{EQN:bRF}) (or equivalently a supersolution of the Ricci flow equation) is called a super Ricci flow. Therefore, it can be expressed as
\[
\frac{\partial g}{\partial \tau} \le 2 \Ric (g(\tau)).
\]
\par The dynamic heat flow of measures $\mu(\tau)$ corresponds to the solutions  (in the distributional sense) of 
\[
\frac{\partial \upmu(\tau)}{ \partial \tau } = \Updelta_{g(\tau)}  \upmu(\tau).
\]
In particular, by standard arguments, the smooth solutions (under smooth Ricci flow) are given by smooth $n$-form $\omega(\tau)$, the density function of which, satisfies the conjugate heat equation McCann-Topping~\cite{MT},
\begin{align*}
	\frac{\partial u}{ \partial \tau } =  \Updelta_{g(\tau)} u - \scal_{g(\tau)}u.
\end{align*}
\par It is verified in McCann-Topping~\cite{MT} that a smooth one-parameter family of metrics is a super Ricci flow if and only if the $\mathcal{L}^2$-Wasserstein distance between two diffusions is non-increasing in $\tau$; namely, satisfying the so-called dynamic coupled contraction property is equivalent to being a super Ricci flow. The same is true using the $\mathcal{L}^1$-Wasserstein distance instead Topping~\cite{Top2}. In McCann-Topping~\cite{MT}, it is also shown that super Ricci flow is equivalent to non-decreasing best Lipschitz constants of dynamic backward heat solutions.
\par The dynamic coupled contraction for general cost functions (cost functions which are increasing in distance) had been previously shown to hold under the super solutions
\[
\frac{\partial g}{\partial \tau} \le 2 \Ric (g(\tau)) + (\nabla^{\sf sym} \mathscr{X}_\tau)^\flat,
\]
for in-homogeneous diffusions generated by -- $\mathcal{C}^1$ in $\tau$ --  time-dependent diffusion operators $\Del_\tau + \mathscr{X}_\tau$; here, $\mathscr{X}_\tau$ is a smooth vector field Arnaudon-Coulibaly-Thalmaier~\cite[Theorem 4.1]{ACT}. It is important to note that the proof in Arnaudon-Coulibaly-Thalmaier~\cite{ACT} does not provide the equivalence.
\par The characterization of smooth super Ricci flow in terms of coupled contraction for dynamic diffusions lead to the following straightforward generalization of the concept of super Ricci flow (still in the smooth realm). 
\begin{definition}
\par Let $\cost = \left\{\cost_\tau\right\}$ be a one-parameter family of cost functions. $g(\tau)$ is called a $\cost$-\textsf{SRF} whenever for any two (smooth) dynamic diffusions $\upmu_1(\tau)$ and $\upmu_2(\tau)$, the optimal total cost
	\[
	\mathcal{T}_{\cost_\tau} \left(\upmu_1(\tau), \upmu_2(\tau)  \right),
	\]
	is non-increasing in $\tau$. 
\end{definition}
\subsection{Weak super Ricci flows via dynamic convexity of entropy}
\par A theory of ($\mathcal{N}$-) super Ricci flows for time-dependent metric measure spaces has recently been developed in Sturm~\cite{St1}. According to Sturm~\cite{St1}, an ($\mathcal{N}$-) super Ricci flow is characterized by a time-dependent metric measure space (with fix underlying set), on which, the Boltzmann entropy is strongly dynamically ($\mathcal{N}$-) convex; see Sturm~\cite{St1} for further details. 
\subsubsection{\small \bf \textsf{\textit{The non-collapsed $\mathcal{L}^2$ - super Ricci flows}}}
\par Much in the same spirit as McCann-Topping~\cite{MT} in the Riemannian setting, Kopfer-Sturm~\cite{KS} explores the  interplay between the dynamical convexity of the Boltzmann entropy and the coupled contraction property of the $\mathcal{L}^2$-Wasserstein distances for diffusions. It is important to emphasize the significance of the time-dependent methods developed in Kopfer-Sturm~\cite{KS} which provide a general framework for working with time-dependent Dirichlet energy, time-dependent entropy, and its gradient flow. For a more comprehensive bibliography on the time-dependent Dirichlet spaces prior to this, see Lakzian-Munn~\cite{LM2}.
\par Kopfer-Sturm~\cite{KS} introduces a theory of what we will refer to as $\mathcal{L}^2$-weak super Ricci flow. Dynamic coupled contraction for ($\mathcal{N}$-) super Ricci flows has been demonstrated under rather strong regularity assumptions on the time slices such as curvature bounds which may not hold for flows that are going through non-global singularities like neckpinches. 
\par Following Kopfer-Sturm~\cite{KS}, let $\left( X, \dist_t, \m_t  \right)$ be a family of Polish compact metric measure spaces where the Borel measures $\m_\tau$ at different times are mutually absolutely continuous with logarithmic densities (with respect to a fixed measure) that are Lipschitz in time. Additionally, suppose the distances satisfy the non-collapsing property
\[
\left|  \ln \frac{\dist_t(x,y)}{\dist_s(x,y)}  \right| \le {\sf L} |t-s|,
\]
which is referred to as the ``log-Lipschitz'' condition. 
\par Given a time dependent family of Dirichlet forms whose square field operators satisfy the diffusion property and uniform ellipticity with respect to a given background strongly local regular Dirichlet form, and if the square field operators applied to logarithmic densities are uniformly bounded, solutions to the dynamic heat and the dynamic conjugate heat equations (defined weakly in terms of the Dirichlet form) exist and are unique; see Kopfer-Sturm~\cite{KS} for further details.
\par A drawback of the otherwise immaculate and successful approach presented in Kopfer-Sturm~\cite{KS} is the less desirable requirement of mutual absolutely continuous measures with Lipschitz densities as well as the Log-Lipschitz condition that one must impose for things to work. These conditions automatically prohibits neckpinches and other non-global singularities from occurring. 
\par As we will see in ~\textsection\thinspace\ref{sec:wrf}, our definition of weak (super) Ricci flow only utilizes upper and lower derivatives in $t$ of the distance $\dist_t$ so does not require $\ln \dist^2_t$ to be Lipschitz. In our approach, $\dist_t$ are even allowed be pseudo distances.
\begin{remark}
\par For time-dependent spaces equipped with time-dependent Dirichlet forms, one can consider $\cost$-\textsf{WSRF}s as those flows where the optimal total cost of transporting between two dynamic diffusions (w.r.t. the cost function $\cost_\tau$) does not increase along the flow where a dynamic diffusion is taken to be in the sense of Kopfer-Sturm~\cite{KS}. 
	\par However, we will not study such flows in these notes. Instead, we will discuss $\cost$-\textsf{WSRF}s defined using the new construction of dynamic diffusions in these notes. 
\end{remark}
\subsection{Trotter-Kato and Chernoff type limits}\label{subsec:Trotter-Chernoff}
\par Two main tools (intimately related) in using semigroups effectively are approximation criteria and perturbation product formulas.  Approximation theorems tell us when semigroups converge to a limit semigroup given the convergence of their generators, and perturbation product formulas tell us how to multiply two semigroups to get a semigroup that is generated by the sum of the generators of the two original semigroups. 
\par A classical Trotter approximation theorem in the form of a Chernoff product formula (which stems from the Trotter-Kato theory) is the following. 
\begin{theorem}[Trotter-Chernoff product formula Pazy~{\cite[Chapter III, Corollary 5.4]{Paz}}]\label{thm:stc}
\par Let $\mathbf{B}$ be a Banach space.. Suppose $\mathscr{F}: [0,\infty) \to \mathcal{L}(\B)$ with $\mathscr{F}(0) = \ident$ that has exponential growth
	\[
	\|\mathscr{F}(\uprho)\| \le e^{\omega \rho},
	\]
for some $\omega \ge 0$. Suppose $\mathscr{F}'(0)v = \mathcal{A}v$ for every $v \in \mathsf{Dom}(\mathcal{A})$ where $\mathcal{A}$ is the infinitesimal generator of a $\mathcal{C}^0$-semigroup $T(\tau)$ on $\mathbf{B}$. Then, 
\[
T(\tau)x = \lim_{n \to \infty} \mathscr{F}\left(\frac{\tau}{n}\right)^n x,
\]
holds locally uniformly in $\tau$. 
\end{theorem}
\par Various forms of time-dependent analogues of the above Trotter-Chernoff product formula have been worked in the literature. A rather all-encompassing version has appeared in Vuillermot~\cite{V}.
\begin{theorem}[Time dependent Trotter-Chernoff product formula Vuillermot~{\cite{V}}]\label{thm:Vuil}
\par Let $\mathbf{B}$ be a complex Banach space. Suppose $\mathscr{F}_\tau: [0,\infty) \to \mathcal{L}(\mathbf{B})$ satisfies $\mathscr{F}(0) = \ident$ and has (uniform in $\tau$) exponential growth
	\[
	\|\mathscr{F}_\tau(\uprho)\| \le e^{\omega \rho},
	\]
	for some $\omega \ge 0$.
\par Consider the derivative $\mathscr{F}_\tau'(0)$ which is an operator with the domain $\mathsf{D}$ consisting of all $f \in \mathcal{L}^\infty(M)$ for which, the strong limit 
\[
\mathscr{F}_\tau'(0)f := \lim_{s \downarrow 0} \frac{\mathscr{F}_\tau(s)f - \mathscr{F}_\tau(0)f}{s},
\]	
exists for all $\tau \in [\tau_1,\tau_2]$. 	
\par Suppose there exists a dense $\mathbf{D} \subset \mathbf{B}$ and an evolution system $U(\tau, \sigma)_{0 \le \sigma \le \tau \le \Uplambda}$ such that for every $v \in \mathbf{D}$, $U(\tau,\sigma)v$ is $\mathcal{C}^1$ in $\tau$ and $U(\tau,\sigma)v \in \mathsf{Dom}\left( \mathscr{F}_\tau'(0) \right)$ and 
\[
\lim_{\tau \downarrow 0} \sup_{\tau \in (\sigma, \Uplambda) } \left\| \left(\frac{\mathscr{F}_\tau(\rho)	U(\tau,\sigma) - U(\tau,\sigma)}{\rho} - \mathscr{F}_\tau'(0)U(\tau,\sigma)\right)v \right\|  = 0;
\]
in particular, the above is satisfied if $\mathscr{F}_\tau(\rho)$ is uniformly in $\tau$ differentiable at $\rho = 0$. 
Then,
\[
U(\tau,\sigma) = \lim_{m \to \infty} \prod_{\ell = m-1}^{0} \mathscr{F}_{\sigma + \ell \left(\frac{\tau - \sigma}{m}\right)}\left( \frac{\tau - \sigma}{m} \right).
\]
\end{theorem}
\begin{remark}
\par Theorem~\ref{thm:Vuil} is stated for complex Banach spaces. We will only be applying the above theorems in either $\mathcal{F}_b$ of bounded functions or Lebesgue spaces $\mathcal{L}^p$. Now since the complex valued bounded functions or the complex valued Lebesgue spaces contain isometric copies of the real valued ones, we are allowed to apply Theorem~\ref{thm:Vuil} in our setting. Alternatively, upon examining the proof in Vuillermot~\cite{V}, one observes that the same proof works for real Banach spaces. 
\end{remark}

\subsection{Weak scalar curvature bounds}
Throughout these notes, \textsf{psc} will mean ``positively scalar curved'' i.e. with positive scalar curvature. Based on the standard asymptotic formula
\begin{align}\label{eq:scal}
	\frac{\mathcal{H}^{n} \left( B^{\tau}_{r}(x) \subset M^n \right)}{\mathcal{H}^{n} \left( B_0(r) \subset \R^n \right)}  = 1 - \frac{r^2}{6(n+2)} \scal_{g(\tau)}(x) f(x) + o(r^4),
\end{align}
(the error term depends locally uniformly on $x$). One deduces \textsf{psc} implies 
\[
\mathcal{H}^{n}(B_r(x)) \le \omega_n r^n,
\]
for $r$ sufficiently small (depending locally uniformly on $x$); c.f. Gromov~\cite[page 7]{Gr}.
\par This motivates us to make the following metric and measure definition that is weaker than \textsf{psc}. 
\begin{definition}[Virtually psc]
	A metric measure space $\left(X, \dist_X, \m  \right)$ is said to be \textsf{virtually psc} if there exists an $r_0(x)>0$ such that
	\[
	\m\left( B_r(x) \right) \le \omega_nr^n, \quad \forall r < r_0(x).
	\]
where $r_0(x)$ can be chosen locally uniformly in $x$. 
\end{definition}
From \eqref{eq:scal}, it is straightforward to see that in the smooth setting, \textsf{virtually psc} is stronger than having non-negative scalar curvature and weaker than \textsf{psc}. 

\par A time-dependent family of metric and measure spaces (resp. Riemannian metrics) is called a \textsf{virtually psc} (resp. \textsf{psc}) flow if each time slice is so and if $r_0$ can be chosen locally uniformly in time as well. 
\section{Static and Dynamic propagators as Trotter-Chernoff type limits}
\par As we previously mentioned, the dynamic diffusions can be rigorously defined using Dirichlet form theory; Kopfer-Sturm~\cite{KS}. But one needs the log-Lipschitz condition. 
\par In these notes, we aim to explicitly construct solutions of dynamic heat and conjugate heat equations using semi-group theory. This approach aligns with the one presented in Kopfer-Sturm~\cite{KS} in the smooth setting but in singular spaces, and without assuming a Dirichlet energy, it gives rise to a formal theory of dynamic heat solutions and dynamic diffusions. So our construction has the added benefit of being applicable to very general metric and measure spaces without requiring further conditions.
\subsection{The building blocks}
\par In Section 4.1 of von Renesse-Sturm~\cite{VS}, it is shown how the heat kernel of a fixed, Riemannian manifold $(M,g)$ may be obtained by applying a Trotter-Chernoff approximation formula to a limit of a family of Markov sphere or ball averaging operators which are defined using only metric-measure properties of $M$. In Lakzian-Munn~\cite{LM1}, a generalization of this construction for a time-dependent family on the disjoint union of two smooth flows in a suitable metric space ambient was presented. 
\par Let $g(\tau)$ be a time-dependent family of smooth Riemannian metrics on $M$. Define a time-dependent family of Markov operators $\upsigma_r$ and $\upnu_r$ acting on the set $\mathcal{F}_b$ of bounded Borel measurable functions  by 
\[
\upsigma^\tau_rf(x) : =  \int_M f(y)~ \dif\s^{\tau}_{r,x}(y), \quad \text{and} \quad \upnu^\tau_rf(x) = \int_M f(y) \dif\bal^\tau_{r,x}(y),
\]
where the normalized surface measures $\s^\tau_{r,x}$ are given by
\[
\s^{\tau}_{r,x} (A) := \frac{\Haus^{n-1}(A \cap \partial B^{\tau}_{r}(x))}{\Haus^{n-1}(\partial B^{\tau}_{r}(x))}, \text{ for Borel measurable sets } A \subset M
,
\]
and the normalized volume measures  $\bal^\tau_{r,x}$ are given by
\[
\bal^{\tau}_{r,x} (A) := \frac{\Haus^{n}(A \cap B^{\tau}_{r}(x))}{\Haus^{n}(B^{\tau}_{r}(x))}, \text{ for Borel measurable sets } A \subset M.
\]
Let $\uptheta^\tau_r$ and $\upeta^\tau_r$ be the linear operators
\begin{align*}
	\uptheta^\tau_r f(x) := \frac{\mathcal{H}^{n-1} \left(\partial B^{\tau}_{r}(x) \subset M^n \right)}{\mathcal{H}^{n-1} \left(\partial B_0(r) \subset \R^n \right)} \Id\left( f(x)\right),
\end{align*}
\begin{align*}
	\upeta^\tau_r f(x) := \frac{\mathcal{H}^{n} \left( B^{\tau}_{r}(x) \subset M^n \right)}{\mathcal{H}^{n} \left( B_0(r) \subset \R^n \right)} \Id\left( f(x)\right).
\end{align*}
\par From standard Riemannian geometry, we have the following asymptotic expansions in $r$ as $r \downarrow 0$, of these operators acting on an $f \in \mathcal{C}^3((M, g(\tau))$.
\begin{align}\label{eq:expansions}
	&\upsigma^\tau_rf(x) = f(x) + \frac{r^2}{2n}\Updelta_{g(\tau)}f(x) + o(r^2); \notag \\
	&\upnu^\tau_rf(x) = f(x) + \frac{r^2}{2(n+2)}\Updelta_{g(\tau)}f(x) + o(r^2); \\
	&\uptheta^\tau_r f(x) = f(x) - \frac{r^2}{6n} \scal_{g(\tau)}(x) f(x) + o(r^4); \notag \\
	&\upeta^\tau_r f(x) = f(x) - \frac{r^2}{6(n+2)} \scal_{g(\tau)}(x) f(x) + o(r^4). \notag
\end{align}
\par Note that the error terms as $r\downarrow 0$ are uniform in $f$ and locally uniform in $x$ since by standard Jacobi field arguments, the error terms are controlled by the local curvature bound and the bound on the first few derivatives (not more than $4$) of the curvature tensor. 
\par Note that we clearly have $\|\sigma^\tau_r\|, \|\upnu^\tau_r\| \le 1$ since these are averaging operators. 
\subsection{The static heat propagator}

\par For a fixed $t \in [s,T]$, the Laplacian $\Updelta_{g(t)}$ is primarily only defined on ${\mathcal{C}}^\infty(M)$ generating a ${\mathcal C}_0$-contractive semigroup (referred to as the static heat flow). Due to the existence of a heat kernel on any manifold (e.g. Chavel~\cite{Cha}), $\Updelta_{g(t)}$ generates a ${\mathcal C}_0$-contractive semigroup in the Banach space $\mathcal{L}^\infty(M)$ with a core ${\mathcal C}^\infty(M)$.

\par For a fixed $t$, and on the core consisting of analytic vectors, the static heat propagator (from time $s_1$ to $s_2$) is given by
\[
\mathcal{H}_{s_1, s_2}^{t} f(x) := e^{ \left( s_2-s_1 \right) \Updelta_{g(t)}} f(x), \quad s_1 \le s_2.
\]

\begin{lemma}\label{lem:unif-diff}
\par Suppose $f_t(x)$ is a smooth dynamic heat solution with smooth initial condition $f_{0}(x)$. Then for any $s \le T$, we have
\[
\lim_{\rho \downarrow 0}	
	\sup_{t \in (s,T)}\left\|\frac{\mathcal{H}^t_{0,\rho} - \ident}{\rho} f_t - \Delta_{g(t)} f_t   \right\| = 0.
	\]
\end{lemma}
\begin{proof}[\footnotesize \textbf{Proof}]
\par By the PDE theory of dynamic heat flows we know $f_t$ is smooth for $t>0$. 
\par By Taylor's theorem with remainder for semigroups, and by the commutativity $\Delta^2_{g(t)}\mathcal{H}^t_{0,\theta} = \mathcal{H}^t_{0,\theta}\Delta^2_{g(t)}$, we obtain
\begin{align*}
\left\|\frac{\mathcal{H}^t_{0,\rho} - \ident}{\rho} f_t - \Delta_{g(t)} f_t   \right\| &\le \frac{\rho}{2} \sup_{0\le \theta \le \rho} \| \Delta^2_{g(t)}\mathcal{H}^t_{0,\theta}f_t \| \\ &= \frac{\rho}{2} \sup_{0\le \theta \le \rho} \| \mathcal{H}^t_{0,\theta} \Delta^2_{g(t)} f_t \|\\
&\le \frac{\rho}{2} \| \Delta^2_{g(t)} f_t \|,
\end{align*}
where the last inequality is due to the contraction property of the static heat flow. 
\par By the smoothness of $f_t$ and the compactness of $M$, we deduce $\sup_{t\in [s,T]} \| \Delta^2_{g(t)} f_t \|$ is bounded for $t \in [s,T]$ for any $s \le T$. Therefore, we get the desired uniform limit. 
\end{proof}
\begin{theorem}[Static heat flow]\label{thm:static-h}
	\par Let $\left( M, g(t)  \right)$ be a time-dependent family of complete Riemannian manifolds. For a fixed $t$, the static heat propagator $
	\mathcal{H}_{s_1, s_2}^{t}$
	is given by the limits
	\begin{align}\label{eq:for-st}
		\mathcal{H}_{s_1,s_2}^{t} f(x) = \lim_{j \to \infty} \left( \upsigma^{t}_{\sqrt{\frac{2n}{j}\left( s_2 - s_1 \right)}}   \right)^j f(x)   = \lim_{j \to \infty} \left( \upnu^{t}_{\sqrt{\frac{2(n+2)}{j}\left( s_2 - s_1 \right)}}   \right)^j f(x),
	\end{align}
	where the convergence is uniformly in $x$ and locally uniformly in $s$. The operator convergence is in the strong operator topology of $\mathcal{L}(\mathcal{F}_b)$. 
\end{theorem}
\begin{proof}[\footnotesize \textbf{Proof}]
	\par To prove the representation of the static flow, set $\mathscr{F}(\rho) := \upsigma^t_{\sqrt{2n\rho}} $ ($t$ is fixed) and note that $\mathscr{F}(0)= \ident$. Since $\|\upsigma^\tau_{\sqrt{2n\rho}}\|\le 1$ , upon applying the Theorem~\ref{thm:stc}, one gets \eqref{eq:for-st}. 
\end{proof}
\subsection{The dynamic heat flow under under a flow of metrics}
\par The standard semigroup theory for evolution equations implies existence of an evolution system for the dynamic heat equation (see Pazy~\cite[Theorem 5.1]{Paz}) on $\mathcal{L}^\infty(M)$. By the construction of the evolution system and the continuity of the generators $\Updelta_{g(t)}$, one deduces the evolution system is continuous in both $t$ and $s$ in the strong operator topology. It is also unique and $\mathcal{C}^1$ in initial and terminal times (Pazy~\cite[Theorem 5.2]{Paz}). This evolution system is known as the dynamic heat propagator and, following the notation of Kopfer-Sturm~\cite{KS}, we denote it by $\mathcal{P}_{s_1,s_2}$. In other terms, $\mathcal{P}_{s_1, s_2} f(x)$ solves the dynamic heat flow (along a smooth Ricci flow) with the initial condition (at time $s_1$) $f(x)$. 
\par Applying the maximum principle, it is clear that $\|\mathcal{P}_{s_1, s_2}\| \le 1$.
\begin{theorem}[Dynamic heat flow]\label{thm:dynamic-h}
 \par The dynamic heat propagator $\mathcal{P}_{s_1, s_2}$ is then given by the following Trotter-Chernoff type product formulae
	\[
	\mathcal{P}_{s_1, s_2} f (x) = \lim_{m \to \infty} \prod_{\ell = m-1}^0 \; {\mathcal H}_{0, \frac{s_2 - s_1}{m}}^{s_1 + \ell\frac{\left( s_2 - s_1  \right)}{m}} f(x), \quad s_1 \le s_2,
	\]
which in turn implies the double limit formulation
	\begin{align}
		\mathcal{P}_{s_1, s_2} f (x) &= \lim_{m \to \infty} \prod_{\ell = m-1}^0 \; \lim_{j \to \infty} \left( \upsigma^{s_1 + \ell\frac{\left( s_2 - s_1  \right)}{m}}_{\sqrt{\frac{2n}{j}\frac{\left( s_2 - s_1 \right)}{m}}}   \right)^j f(x) \notag \\ &=  \lim_{m \to \infty} \prod_{\ell = m-1}^0 \; \lim_{j \to \infty} \left( \upnu^{s_1 + \ell\frac{\left( s_2 - s_1  \right)}{m}}_{\sqrt{\frac{2(n+2)}{j}\frac{\left( s_2 - s_1 \right)}{m}}}   \right)^j f(x). \notag
	\end{align}
Alternatively, the dynamic heat flow can be obtained via the single limit formula
\begin{align*}
	\mathcal{P}_{s_1, s_2} f (x) &= \lim_{m \to \infty}\prod_{\ell = m-1}^0 \upsigma^{s_1 + \ell \left( \frac{s_2-s_1}{m} \right)}_{\sqrt{2n\left( \frac{s_2-s_1}{m} \right)}}\\
	&= \lim_{m \to \infty}\prod_{\ell = m-1}^0 \upnu^{s_1 + \ell \left( \frac{s_2-s_1}{m} \right)}_{\sqrt{2(n+2)\left( \frac{s_2-s_1}{m} \right)}}.
\end{align*}
\end{theorem}
\begin{proof}[\footnotesize \textbf{Proof}]
\par For the double limit representation of the dynamic heat flow, take $\mathscr{I}_t(\rho) := \mathcal{H}^t_{0,\rho}$ with $\mathscr{I}_t(0) = \ident$ and notice $\mathscr{I}_t'(0) = \Delta_{g(t)}$. The other requirements of Theorem~\ref{thm:Vuil} follows from the standard properties of the dynamic heat flow and its corresponding system. Based on Lemma~\ref{lem:unif-diff} and applying Theorem~\ref{thm:Vuil}, we get the conclusion for $f$ in the core $\mathcal{C}^\infty(M)$. 
\par To prove the single limit representation, let $\mathscr{J}_t(\rho) := \upsigma^t_{\sqrt{2n\rho}}$ and apply Theorem~\ref{thm:Vuil}; for more details, see the proof of the single limit representation in Lakzian-Munn~\cite{LM1}. 
\end{proof}
\begin{remark}
\par Indeed, the representation of the static heat and dynamic heat flow as in the Theorem~\ref{thm:static} also follows from more classical approximation theorems than the ones given in~\textsection\thinspace\ref{subsec:Trotter-Chernoff}. Since the semigroups involved are contraction semigroups, one can apply the contractive approximation theorems such as the one in Ethier-Kurtz~\cite[Chapter 1, Theorem 6.5]{EK} for the static case and use the Faris~\cite{Faris} for the dynamic case. This way, one will not need the Lemma~\ref{lem:unif-diff}. 
\end{remark}
\subsection{Conjugate heat flow under the Ricci flow}\label{subsec:chf}
\par According to the bounded perturbation theorem (see Engel-Nagel~\cite[Theorem 1.3]{EN}), the static (i.e. $\tau$ is kept fixed) conjugate heat generator under the Ricci flow 
\[
\mathcal{CH}^{\tau}:=\Updelta_{g(\tau)} - \scal_{g(\tau)} \ident,
\]
generates a $\mathcal{C}_0$ semigroup on ${\mathcal{L}}^\infty(M)$  with a core ${\mathcal C}^\infty(M)$. We denote this semigroup (or indeed its corresponding static evolution system) by $\mathcal{CH}^{\tau}_{s_1,s_2}$ where $s_1$ is the start time of the flow and $s_2$ denotes the stop time. 
\par Moreover, following the standard theory of evolution semigroups, we obtain an evolution system known as the dynamic conjugate heat propagator. Following the notations of Kopfer-Sturm~\cite{KS}, we denote it by $\mathcal{P}^\star_{\tau_1, \tau_2}$. In other words, $\mathcal{P}^\star_{\tau_1, \tau_2} f(x)$ solves the conjugate heat flow (along a smooth Ricci flow) with the initial condition $f(x)$. 
\par Recall that the dynamic heat and conjugate heat flows (along any flow of metrics) are dual via the identity 
	\begin{align}\label{eq:duality}
	\mathcal{P}^\star\left(  \tau_0, \tau \right) g(y) := \int_M \;  \mathcal{P}_{\tau, \tau_0}\mathbf{\delta_y}(x)   \; g(x) \; \dvol_{g(\tau)}, \quad  \tau_0 \le \tau,
	\end{align}
	where
$P_{\tau, \tau_0}\mathbf{\delta_y}(x) $ is the delta function propagated under heat flow which is nothing but the dynamic heat transition kernel.	
\subsubsection{Static conjugate heat propagator}
\par Consider the operators
\begin{align}\label{eq:alpha-beta}
\upalpha^\tau_r := \frac{1}{4} \upsigma^\tau_r + \frac{3}{4} \uptheta^\tau_r, \quad \upbeta^\tau_r := \frac{1}{4} \upnu^\tau_r + \frac{3}{4} \upeta^\tau_r.
\end{align}
Using \eqref{eq:expansions}, we obtain the asymptotic expansions
\[
\upalpha^\tau_r f(x) = f(x) + \frac{r^2}{8n} \left( \Updelta_{g(\tau)} - \scal_{g(\tau) }(x)\right) f(x) + o(r^2),
\]
and
\[
\upbeta^\tau_r f(x) = f(x) + \frac{r^2}{8(n+2)} \left( \Updelta_{g(\tau)} - \scal_{g(\tau) }(x)\right) f(x) + o(r^2)
\]
as $r \downarrow 0$.
\par Expressing the dynamic conjugate heat propagator as a Trotter-Chernoff type product of operators expressed in terms of metric and measure is more challenging due to the inhomogeneous nature of its generator. For now, we will only be able to deal with the case where the scalar curvature is positive along the Ricci flow. 
\par Note that since $M$ is assumed compact, by the evolution equation of scalar curvature under Ricci flow
\[
\label{eqn:sc-evol-RF}
\frac{\partial \scal_{g(\tau)}}{ \partial \tau } =  - \Updelta_{g(\tau)}\scal_{g(\tau)} - \|\Ric_{g(\tau)}\|^2,
\]
the scalar curvature is positive along the flow (on the interval $[\tau_1,\tau_2]$) if and only if $\scal_{g(\tau_2)} >0$. 
\begin{lemma}\label{lem:growth-conj}
	Suppose $g(\tau)$ is a Ricci flow with $\scal_\tau \ge 0$. Then 
\[
 \|\mathcal{CH}^{\tau}_{0,\rho}\|\le 1.
\]
If $g(\tau)$ is furthermore a \textsf{virtually psc} Ricci flow (in particular, if it is \textsf{psc}), then
\[
\|\alpha^\tau_r\|, \|\upbeta^\tau_r\| \le 1,
\]
as well for sufficiently small $r>0$ (locally uniform in $\tau$). 
\end{lemma}
\begin{proof}[\footnotesize \textbf{Proof}]
\par The first claim is by maximum principle. Suppose $\|f\|\le 1$. Suppose for some $\rho>0$, $\mathcal{CH}^{\tau}_{0,\rho}f$ achieves a positive maximum at some point $x_0$. Then
\[
\partial_\rho \mathcal{CH}^{\tau}_{0,\rho}f(x_0) = \left(\Updelta_{g(\tau)} - \scal_{g(\tau)}\right) f(x_0) \le 0,
\]
due to the fact that $\Updelta_{g(\tau)}f(x_0)\le 0$ at a maximum point, and also by nonnegativity of the scalar curvature we have $-\scal_{g(\tau)} f(x_0) \le 0$. This indicates that the positive maximum never increases along the flow and similarly, negative minimums never decreases. Thus, we conclude that
\[
\|\mathcal{CH}^{\tau}_{0,\rho}f\|\le \|f\|,
\]
which implies $\|\mathcal{CH}^{\tau}_{0,\rho}\|\le 1$. 
\par Assuming $g(\tau)$, $\tau_1 \le \tau \le \tau_2$ is a \textsf{virtually psc} Ricci flow, and since $M$ is compact, there exists a uniform (in $x$) $r_0>0$ such that for $r\le r_0$, we have $\|\uptheta^\tau_r\|, \|\upeta^\tau_r\|\le 1$. Consequently, thus from the definition, we also obtain $\|\upalpha^\tau_r\|$, $\|\upbeta^\tau_r\| \le 1$. Note that the error terms in the expansions \eqref{eq:expansions} are governed by the curvature bounds (using standard Jacobi field arguments) and since $M$ is compact, one deduces that such $r_0>0$ can be chosen locally uniformly in $\tau$. 
\end{proof}
\begin{theorem}[''Static'' conjugate heat propagator]
	Suppose $g(\tau)$ is a \textsf{virtually psc} Ricci flow (in particular, if it is \textsf{psc}). Then,
	\[
	\mathcal{CH}^\tau_{\tau_1, \tau_2} f (x) = \lim_{j \to \infty} \left( \upalpha^\tau_{\sqrt{8n\frac{\left(\tau_2 - \tau_1 \right)}{j}}}   \right)^j f(x) = \lim_{j \to \infty} \left( \upbeta^\tau_{\sqrt{8(n+2)\frac{\left(\tau_2 - \tau_1 \right)}{j}}}   \right)^j f(x).
	\]
for $f \in \mathcal{F}_b$.
\end{theorem}
\begin{proof}[\footnotesize \textbf{Proof}]
	\par  For a fixed $\tau$, set
$
\mathscr{K}(s) := \upalpha^{\tau}_{\sqrt{8ns}}
$.
Notice that $\mathscr{K}(0) = \ident$. 
Recall that the derivative at $s=0$, $\mathscr{K}'(0)$ is an operator with the domain $\mathsf{D}$ which consists of all $f \in \mathcal{L}^\infty(M)$ for which, the strong limit 
\[
\mathscr{K}'(0)f := \lim_{s \downarrow 0} \frac{\mathscr{K}(s)f - \mathscr{K}(0)f}{s}, 
\]	
exists for all $\tau \in [\tau_1,\tau_2]$. Therefore, it is clear that $\mathcal{C}^\infty(M) \subset \mathsf{Dom}\left( \mathscr{K}'(0) \right)$.
\par By Lemma~\ref{lem:growth-conj}, we get the required growth rate necessary to apply Theorem~\ref{thm:stc}. Hence the conclusion holds for elements in the core $\mathcal{C}^\infty$. 
\par The second identity is proven in the same manner. 
\end{proof}
\subsubsection{Dynamic conjugate heat propagator}
\begin{lemma}\label{lem:unif-dif-2}
\par Suppose $g(\tau)$ is a Ricci flow on $[\tau_1,\tau_2]$ with $\scal_\tau\ge 0$. Suppose $f_t(x)$ is a smooth dynamic conjugate heat solution with a smooth initial condition $f_{0}(x)$. Then for any $\tau_1 \le \tau_2$, we have
	\begin{align}\label{eq:unif-dif-2}
	\lim_{\rho \downarrow 0}	
	\sup_{\tau \in (\tau_1,\tau_2)}\left\|\frac{\mathcal{CH}^\tau_{0,\rho} - \ident}{\rho} f_\tau - \left( \Updelta_{g(\tau)} - \scal_{g(\tau)} \right) f_\tau   \right\| = 0.
	\end{align}
\end{lemma}
\begin{proof}[\footnotesize \textbf{Proof}]
\par Similar to Lemma~\ref{lem:unif-diff}, we can apply Taylor's theorem with remainder for semigroups. Additionally, utilizing the commutativity 
\[
\left( \Updelta_{g(\tau)} - \scal_{g(\tau)} \right)^2\mathcal{CH}^\tau_{0,\theta} = \mathcal{CH}^\tau_{0,\theta}\left(\Updelta_{g(\tau)} - \scal_{g(\tau)} \right)^2,
\]
we get
	\begin{align*}
		\left\|\frac{\mathcal{CH}^\tau_{0,\rho} - \ident}{\rho} f_\tau - \left(\Updelta_{g(\tau)} - \scal_{g(\tau)}  \right) f_\tau   \right\| &\le \frac{\rho}{2} \sup_{0\le \theta \le \rho} \| \left(\Updelta_{g(\tau)} - \scal_{g(\tau)} \right)^2\mathcal{CH}^\tau_{0,\theta}f_\tau \| \\ &= \frac{\rho}{2} \sup_{0\le \theta \le \rho} \| \mathcal{CH}^\tau_{0,\theta} \left(\Updelta_{g(\tau)} - \scal_{g(\tau)} \right)^2 f_\tau \|\\
		&\le \frac{\rho}{2} \| \left(\Updelta_{g(\tau)} - \scal_{g(\tau)} \right)^2 f_\tau \|,
	\end{align*}
	where the last inequality follows from the contraction property of static heat flow. 
\par Due to the smoothness of $f_\tau$ and the compactness of $M$, we can infer that 
\[
\sup_{\tau\in [\tau_1,\tau_2]} \| \left(\Updelta_{g(\tau)} - \scal_{g(\tau)} \right)^2 f_\tau \|,
\]
 is bounded for any $\tau_1 \le \tau_2$. Therefore, we achieve the desired uniform limit. 
\end{proof}
\begin{theorem}[Forward in $\tau$ conjugate heat flow under backward Ricci flow]\label{thm:st-dyn-conj}
\par Let $\left( M, g(\tau)  \right)$ be a backward Ricci flow with $\scal_\tau \ge 0$. The conjugate heat propagator $\mathcal{P}^\star_{\tau_1, \tau_2}$ is given by the single limit
	\[
	\mathcal{P}^\star_{\tau_1, \tau_2} f(x) = \lim_{m \to \infty} \prod_{\ell = m-1}^0 \; \mathcal{CH}_{0, \frac{\tau_2 - \tau_1}{m}}^{\tau_1 + \ell\frac{\left( \tau_2 - \tau_1  \right)}{m}} f(x), \quad \tau_1 \le \tau_2,
	\]
	where $\mathcal{CH}^\tau_{\tau_1, \tau_2}$ denotes the static conjugate heat propagator 
	\[
	\mathcal{CH}^\tau_{\tau_1, \tau_2} f (x) := e^{\left( \tau_2 - \tau_1  \right) \left(\Updelta_{g(\tau)} - \scal_{g(\tau)}\text{$\ident$} \right)} f(x), \quad \tau_1 \le \tau_2.
	\]
	For fixed $\tau$. The convergence is uniformly in $x$ and locally uniformly in $\tau_i$. The operators also converge in the strong operator topology of $\mathcal{L}(L^\infty(M))$. 
\par Alternatively, if the flow is furthermore a \textsf{virtually psc} flow (in particular, if it is \textsf{psc}), $\mathcal{P}^\star_{\tau_1, \tau_2}$ can be represented by the double limit
\begin{align*}
\mathcal{P}^\star_{\tau_1, \tau_2} f(x) &= \lim_{m \to \infty} \prod_{\ell = m-1}^0 \; \lim_{j \to \infty} \left( \upalpha^{s_1 + \ell\frac{\left( s_2 - s_1  \right)}{m}}_{\sqrt{\frac{8n}{j}\frac{\left( s_2 - s_1 \right)}{m}}}   \right)^j f(x)\\
&= \lim_{m \to \infty} \prod_{\ell = m-1}^0 \; \lim_{j \to \infty} \left( \upbeta^{s_1 + \ell\frac{\left( s_2 - s_1  \right)}{m}}_{\sqrt{\frac{8(n+2)}{j}\frac{\left( s_2 - s_1 \right)}{m}}}   \right)^j f(x).
\end{align*}
\end{theorem} 
\begin{proof}[\footnotesize \textbf{Proof}]
For	the single limit representation (for core elements $f \in \mathcal{C}^\infty(M)$), set $\mathscr{Q}_\tau(\rho) := \mathcal{CH}^\tau_{0, \rho}$ and apply Theorem~\ref{thm:Vuil} on the basis of Lemmas~\ref{lem:unif-dif-2} and \ref{lem:growth-conj}. The other requirements hold by the PDE theory of the conjugate heat equation under Ricci flow and the resulting evolution system.
\par The double limit representation is proven similarly by defining the map
$
\mathscr{Q}_\tau(\rho) := \upalpha^\tau_{\sqrt{8n\rho}}, 
$
and applying Theorem~\ref{thm:Vuil} on the basis of Lemmas~\ref{lem:growth-conj} and \ref{lem:unif-dif-2}. 
\end{proof}
\begin{remark}
\par Using the older time-dependent Trotter-Chernoff limit as in Faris~\cite{Faris} (since we are working with contractions here as shown in Lemma~\ref{lem:growth-conj}), one directly verifies the conclusion of Theorem~\ref{thm:st-dyn-conj} for $f \in \mathcal{F}_b$ directly without needing to verify the uniform limits~\ref{eq:unif-dif-2}.  
\end{remark}
\begin{remark}
\par The universal coefficients $\nicefrac{1}{4}$ and $\nicefrac{3}{4}$ in \eqref{eq:alpha-beta} reveal an interesting phenomenon. In the conjugate heat flow under a \textsf{virtually psc} backward Ricci flow, the concentration caused by ``$-\scal_{g(\tau)}$'' term, occurs -- in a sense -- three times faster than the diffusion caused by $\Updelta_{g(\tau)}$. This is an intriguing characteristic of diffusions under Ricci flow to take note of.
\end{remark}
\section{Formal dynamic heat flow and diffusions in the singular setting}
\par Based on the characterization of dynamic diffusions in the smooth setting, one can define operator-theoretic dynamic diffusions along a flow by directly using the construction outlined in the previous sections. However, in general metric measure spaces, only volumes of geodesic balls are well-behaved in general and not the sphere areas. Therefore, we will only use the iterates of the operators $\upnu^\tau_r$ and $\upbeta^{\tau}_r$. 
\begin{definition}
\par We say that a time-dependent metric space $\left(X,\dist_\tau\right)$ has Hausdorff dimension $n$ when for each $\tau$, $\dim_{\mathsf{Haus}}(X,\dist_\tau) = n$. 
\end{definition}
\begin{definition}[Formal static and dynamic heat flow]\label{defn:form-static-op}
\par Suppose $\overrightarrow{\mathcal{X}} = \left( X, \dist_t, \m_t\right)$ is a time-dependent metric-measure space with Hausdorff dimension $n$. We define the formal static and dynamic heat propagators as
	\[
	\textsf{formal}-\mathcal{H}_{s_1,s_2}^{t} f(x) := \lim_{j \to \infty} \left( \upnu^{t}_{\sqrt{\frac{2(n+2)}{j}\left( s_2 - s_1 \right)}}   \right)^j f(x),
	\]
	\[
	\textsf{formal}-\mathcal{P}_{s_1, s_2} f (x) = \lim_{m \to \infty} \prod_{i = m-1}^0 \; \textsf{formal}-{\mathcal H}_{0, \frac{s_2 - s_1}{m}}^{s_1 + \frac{i}{m}\left( s_2 - s_1  \right)} f(x), \quad s_1 \le s_2,
	\]
	where the domain consists of those $f \in \mathcal{F}_b(s_1)$ where the RHS limit exists in $\mathcal{F}_b(s_2)$.
\end{definition}
\begin{definition}[Formal heat solutions along singular flows]
\par In the context of Definition~\ref{defn:form-static-op}, a static formal heat solution is a function $f_s(x)$ with
\[
f_s \in \mathsf{Dom}\left( \textsf{formal}-\mathcal{H}_{s,s_2}^{t} \right),
\]
($t$ is fixed) for any $s_2 \ge s$ and with
\[
\textsf{formal}-\mathcal{H}_{s_1,s_2}^{t} f_{s_1}(x) = f_{s_2}(x),
\] 
for all $s_1 \le s_2$. Dynamic heat solutions are defined in a similar fashion. 
\end{definition}
\begin{definition}[Formal static and dynamic conjugate heat operator]\label{defn:formal-stat-conj-op}
\par For a given time-dependent metric measure space $\overleftarrow{\mathcal{X}} = \left( X, \dist_\tau, \m_\tau\right)$ with Hausdorff dimension $n$, we define
	\[
	\textsf{formal}-\mathcal{CH}^\tau_{\tau_1, \tau_2} f (x) := \lim_{j \to \infty} \left( \upbeta^\tau_{\sqrt{8(n+2)\frac{\left(\tau_2 - \tau_1 \right)}{j}}}   \right)^j f(x).
	\]
The domain consists of all functions for which the RHS strong limit exists. The dynamic formal conjugate heat propagator is defined as 
\begin{align*}
	\textsf{formal}-\mathcal{P}^\star_{\tau_1, \tau_2} f(x) &= \lim_{m \to \infty} \prod_{i = m-1}^0 \; \textsf{formal}-\mathcal{CH}_{0, \frac{\tau - \tau_0}{m}}^{\tau_0 + \frac{i}{m}\left( \tau - \tau_0  \right)} f(x).
\end{align*}
\end{definition}
\begin{definition}[Formal dynamic diffusions along singular flows]
\par Assume the setting of Definition~\ref{defn:formal-stat-conj-op}. Let $\left(X, \dist_\tau, \m_\tau  \right)$, $\tau \in [\tau_1,\tau_2]$ be a sequence of metric-measure spaces.  A family of Borel probability measures $\mu(\tau) = f_\tau(x)\m_\tau $ (noticing that the topology might change as $\tau$ varies and consequently so does the Borel $\sigma$-algebra) is called a Ricci flow like dynamic diffusion when $f(\tau)$ satisfies
	\[
		f_\tau = \textsf{formal}-\mathcal{P}^\star_{\tau_0, \tau} f_{\tau_0}, 
	\]
for all $\tau_0 \le \tau$. 
\end{definition}
\begin{remark}
\par Notice these are just formal heat flows and conjugate heat flows that do not involve any energy functionals involved and are modeled on $n$-dimensional manifolds. However, on Riemannian manifolds, they coincide with the standard dynamic heat and conjugate heat flows. 
\end{remark}
\begin{remark}
	In the definition of formal diffusions, we want to allow for the initial measure (at time $\tau_1$) to be a delta measure even though it is not absolutely continuous. In this case we set $f_{\tau_1} = \delta_x$. 
\end{remark}
\section{Metric-measure formulation of super Ricci flows}
\par Here we present two definitions for weak super Ricci flows. One uses formal dynamic heat flows (\textsf{WSRF}) and the second one uses coupled contraction of formal dynamic diffusions in \textsf{virtually psc} time-dependent spaces. 
\par In the latter definition, the contraction is w.r.t. a time-dependent family of cost functions (so these flows are called $\cost$-\textsf{WSRF}s); consequently, this adds another layer of generalization. 
\par The classical Ricci flow then corresponds to the case where the cost is distance squared. 
\par Let $\overrightarrow{\mathcal{X}} = \left( X, \dist_t, \m_t\right)$ be a forward time-dependent metric measure space of Hausdorff dimension $n$. 
\begin{definition}[\textsf{WSRF}]
\par $\overrightarrow{\mathcal{X}}$ is a weak super Ricci flow (\textsf{WSRF} for short) if for all Lipschitz formal heat solutions, the Lipschitz constant is non-increasing in $t$. 
\end{definition}
Recall \emph{a time-dependent metric-measure space is said to be a \textsf{virtually psc} flow if each time section is \textsf{virtually psc} and if $r_0$ can be chosen local uniformly in the time parameter}. 
\par Let $\overleftarrow{\mathcal{X}} = \left(X, \dist_\tau, \m_\tau, \cost_\tau \right)$ be an evolving 4-tuple consisting of a backward time-dependent metric-measure space equipped with a time-dependent cost function. 
\begin{definition}[$\cost$-\textsf{WSRF}]
\par $\overleftarrow{\mathcal{X}}$ is called a $\cost$-\textsf{WSRF} if for all formal dynamic diffusions the coupled contraction holds true i.e. if for any two dynamic diffusions $\upmu_1(\tau)= f_1(\tau)\m_\tau$
	and $\upmu_2(\tau) = f_2(\tau)\m_\tau$, the optimal total cost 
	\[
	\mathcal{T}_{\cost_\tau} \left(\upmu_1(\tau), \upmu_2(\tau)  \right),
	\]
is non-increasing in $\tau$. 
\end{definition}
\begin{remark}
\par In Lakzian-Munn~\cite{LM2}, a more restrictive version of $\cost$-\textsf{WSRF}s is considered based on the dynamic diffusion theory presented in Kopfer-Sturm~\cite{KS} to capture the one point pinching phenomenon in symmetrical neckpinches within the framework of singular flows. 
\end{remark}
\subsection{What are $\cost$-\textsf{WSRF}s in the smooth setting?}
\par Suppose $g(\tau)$ is a family of Riemannian metrics on a compact $M$. Then $\left(M, \dist_\tau, \dvol_\tau, \cost_\tau = \dist_\tau^2\right)$ is a $\cost$-\textsf{WSRF} if and only if $g(\tau)$ is a smooth super Ricci flow if and only if $\left(M, \dist_t, \dvol_t \right)$ is a \textsf{WSRF}. This is indeed the content of McCann-Topping~\cite{MT}. 
\par For more general families of cost functions, the equivalence between these concepts is not clear even for smooth families of metrics. For example deriving dynamic coupled contraction from the super Ricci flow equation utilizes the Monge-Amp\'ere equation which only takes a nice form for the distance squared cost function. 
\par However for cost functions that are monotonically increasing convex functions of the distance, one gets the following. 
\begin{theorem}\label{thm:conv-cost}
\par Let $\cost_\tau = \mathdutchcal{c}_\tau(\dist)$ be a one-parameter family of cost functions so that for each $\tau$,  $\mathdutchcal{c}_\tau$ is a monotonically increasing convex function of $\dist_\tau$. Furthermore, suppose $\cost_\tau$ is continuous in $\tau$. Then, the coupled contraction for smooth dynamic diffusions with respect to $\cost_\tau$, implies super Ricci flow.
\end{theorem}
\begin{proof}[\footnotesize \textbf{Proof}]
\par By the Jensen inequality, 
	\begin{align}\label{eq:Jensen}
	\mathdutchcal{c}_\tau \left( \int\!\!\!\!\int_{M^2}  \dist(x_1,x_2) \dif\mathdutchcal{q}(x_1,x_2)  \right) \le  \int\!\!\!\!\int_{M^2}  \cost_\tau (x_1,x_2) \dif\mathdutchcal{q}(x_1,x_2),
	\end{align}
holds for any probability measure $\mathdutchcal{q}$ on $M^2$. 	
\par Proceeding similarly to McCann-Topping~\cite[Section 5]{MT}, consider two fixed points $p_1$ and $p_2$ and two forward times $a < b$. Suppose $a$ and $b$ correspond to backward times $\tau_a$ and $\tau_b$ respectively; so we have $\tau_b < \tau_a$. 
\par Let $\mu_1(\tau)$ and $\mu_2(\tau)$ be dynamic diffusions with initial conditions $\mu_1(\tau_b) = \delta_{p_1}$ and $\mu_2(\tau_b) = \delta_{p_2}$. For times $\tau > \tau_b$, the diffusions are absolutely continuous w.r.t. $\dvol_\tau$ with the densities $u_1(p_1,\tau_b, x, \tau)$ and $u_2(p_2,\tau_b, y, \tau)$ respectively.  
\par By continuity in $\tau$ and compactness of $M$, we have
\[
\lim_{\tau \to \tau_b} \mathcal{T}_{\cost_\tau}\left( \upmu_1(\tau), \upmu_2(\tau) \right) = \cost_{\tau_b}(p_1,p_2) = \mathdutchcal{c}_{\tau_b}\left(\dist_{\tau_b}(p_1,p_2)\right),
\]
and by the coupled contraction assumption, we have
\[
\mathcal{T}_{\cost_\tau}\left( \upmu_1(\tau), \upmu_2(\tau) \right) \le \mathdutchcal{c}_{\tau_b} \left(\dist_{\tau_b}(p_1,p_2)\right),
\]
or
\begin{align}\label{eq:const-ineq}
\mathdutchcal{c}_{\tau_b}^{-1} \left( \mathcal{T}_{\cost_\tau}\left( \upmu_1(\tau), \upmu_2(\tau) \right) \right) \le \dist_{\tau_b}(p_1,p_2).
\end{align}
\par Let $f_t$ (with corresponding backward time parametrization $f_\tau$) be a dynamic heat solution in forward time. According to the definition of the conjugate heat flow (see~\eqref{eq:duality}), we can express it as
\[
f_{\tau_b}(p_i) = \int_M u_i(p_i,\tau_b, z, \tau)f(z,\tau)\dif z, \quad i = 1,2,
\]
\par Let $\pi(x,y)$ be a transport plan between $\upmu_1(\tau_a)$ and $\upmu_2(\tau_a)$. Since $\pi$ is a measure on $M^2$ with marginals of which are $\upmu_1(\tau_a)$ and $\upmu_2(\tau_a)$, we have
\[
f_{\tau_b}(p_1) = \int\!\!\!\!\int_{M^2} f_{\tau_a}(y)\dif\pi(y,z), \quad f_{\tau_b}(p_2) = \int\!\!\!\!\int_{M^2} f_{\tau_a}(z)\dif\pi(y,z);
\]
thus
\begin{align}\label{eq:lip}
|f_{\tau_b}(p_2) - f_{\tau_b}(p_1)| &\le \int\!\!\!\!\int_{M^2} \left| f_{\tau_a}(z) - f_{\tau_a}(y)  \right| \dif\pi(y,z) \\
&\le \Lip\left(f_{\tau_a} \right)  \int\!\!\!\!\int_{M^2} \dist_{\tau_a}(y,z) \dif\pi(y,z).\notag
\end{align}
\par By \eqref{eq:Jensen}, we have
	\[
	\int\!\!\!\!\int_{M^2} \dist_{\tau_a}(y,z) \dif\pi(y,z)  \le \mathdutchcal{c}_{\tau_b}^{-1} \left( \int\!\!\!\!\int_{M^2} \cost_{\tau_a}(y,z)  \dif\pi(y, z)    \right).
	\]
\par Hence for the optimal transport plan $\pi$, as per \eqref{eq:lip} and \eqref{eq:const-ineq}, we deduce that
	\[
	|f_{\tau_b}(p_2) - f_{\tau_b}(p_1)|  \le  \Lip \left( f_{\tau_a} \right) \mathdutchcal{c}_{\tau_b}^{-1}\left(  \mathcal{T}_{\cost_{\tau_a}} (\upmu_1(\tau_a), \upmu_2(\tau_a) \right) \le   \Lip \left( f_{\tau_a} \right) \dist_{\tau_b}(p_1,p_2),
	\]
	which implies
	$
	\Lip \left(f_{\tau_b} \right) \le \Lip \left( f_{\tau_a} \right).
	$
\end{proof}
\begin{theorem}
\par Let $\cost_\tau = \mathdutchcal{c}_\tau(\dist)$ be as in the previosu theorem. Suppose the one-parameter family $g(\tau)$ of smooth \textsf{virtually psc} metrics is a $\cost$-\textsf{WSRF}. Then, $g(\tau)$ is a super Ricci flow. In particular, this holds true for any \textsf{psc} flow. 
\end{theorem}
\begin{proof}[\footnotesize \textbf{Proof}]
\par Since $g(\tau)$ is \textsf{virtually psc}, by Theorem~\ref{thm:conv-cost} and the characterization in McCann-Topping~\cite{MT}, being $\cost$-\textsf{WSRF} in this setting, implies $\cost$-\textsf{SRF}. According to Theorem~\ref{thm:conv-cost} and the characterization in McCann-Topping~\cite{MT}, it is an \textsf{SRF}.  
\end{proof}
\section{Metric-measure characterization of the smooth Ricci flow}
\par In this section we will present the saturation property which prevents a super Ricci flow from having any deficit in being a Ricci flow. This saturation property is also expressed in terms of metrics and measures, thus completing the metric-measure characterization of Ricci flow. 
\par Recall that $\bal_{x,\eps}^\tau$ denotes the uniform probability measure of the ball $B_x(\varepsilon) \subset M$ and $\s^\tau_{x,\eps}$ denotes the uniform area measure of the geodesic sphere $S_x(\varepsilon) \subset M$, all w.r.t. $g(\tau)$.
\begin{definition}[Saturated flow]
\par Let $g(\tau)$ be a super Ricci flow. We say the flow $g(\tau)$ is saturated if
\begin{align*}
	\lim_{\varepsilon \downarrow 0} \frac{1}{\varepsilon^{2}}\Bigg( 12\left(\frac{\mathcal{H}^{n} \left( B^{\tau}_{\varepsilon}(x) \right)}{\omega_n\varepsilon^n} -1\right) + \int_{\dist_{\tau}(x,y) \le \varepsilon } \left[\partial_\tau \dist^2_{\tau}(x,y)\right]  \;\dif\bal^\tau_{x,\eps}  \Bigg)\ge 0,
\end{align*}
holds for every $x \in M$; or if equivalently
\begin{align*}
	\lim_{\varepsilon \downarrow 0} \frac{1}{\varepsilon^{2}}\Bigg( \frac{12n}{n+2} \left(\frac{\mathcal{H}^{n-1} \left( \partial B^{\tau}_{\varepsilon}(x) \right)}{ a_{n-1}\varepsilon^{n-1}} -1\right) + \int_{\dist_{\tau}(x,y) = \varepsilon } \left[\partial_\tau \dist^2_{\tau}(x,y)\right]  \;\dif\s^\tau_{x,\eps}\Bigg)\ge 0,
\end{align*}
is true at every $x$. 
\end{definition}
\begin{theorem}[Characterization of Ricci flow]\label{thm:Ric-Sat}
\par For a super Ricci flow $g(\tau)$ on a compact manifold $M$, the following are equivalent:
	\begin{enumerate}
		\item $g(\tau)$ is a Ricci flow
		\item $g(\tau)$ is saturated.
	\end{enumerate}
\end{theorem}
\par For the proof, we will need the following key trace formula:
\begin{theorem}[Trace formula under a (super) flow]\label{thm:trace}
	Suppose $\partial_\tau g_\tau \;\; \le (=) \;\; h_\tau$ then
	\begin{align*}
		\lim_{\varepsilon \downarrow 0} \frac{1}{\varepsilon^{2}} \int_{\dist_{\tau}(x,y) \le \varepsilon } \left[\partial_\tau \dist^2_{\tau}(x,y)\right]  \;\dif\bal^\tau_{x,\eps} \;\;\le (=) \;\; \frac{1}{n+2} \tr_{g(\tau)} h_\tau,
	\end{align*}
and also
\begin{align*}
	\lim_{\varepsilon \downarrow 0} \frac{1}{\varepsilon^{2}} \int_{\dist_{\tau}(x,y) = \varepsilon } \left[\partial_\tau \dist^2_{\tau}(x,y)\right]  \;\dif\s^\tau_{x,\eps} \;\;\le (=) \;\; \frac{1}{n} \tr_{g(\tau)} h_\tau.
\end{align*}
\end{theorem}
\begin{proof}[\footnotesize \textbf{Proof}]
	To estimate
\begin{align}\label{eq:limit-trace}
	\lim_{\varepsilon \downarrow 0} \frac{1}{\varepsilon^{2}} \int_{\dist_{\tau}(x,y) \le \varepsilon } \left[\partial_\tau \dist^2_{\tau}(x,y)\right]  \;\dif\bal^\tau_{x,\eps},
\end{align}
we use the change of variable $y = \exp^{\tau}_x(\eps v)$ to get
\begin{align*}
	&\frac{1}{\varepsilon^{2}} \int_{\dist_{\tau}(x,y) \le \varepsilon }  \left[\partial_\tau \dist^2_{\tau}(x,y)\right]  \; \dif\bal^\tau_{x,\eps} \\
	&= \frac{1}{\varepsilon^{2}}\int_{\|v\|_{\tau} \le 1 } \varepsilon^2 \partial_\tau g_{\tau}(v,v)  \; (\exp_x^\tau(\varepsilon \cdot))^*(\dif\bal^\tau_{x,\eps})\\
	&= \int_{\|v\|_{\tau} \le 1 }  \partial_\tau g_{\tau}(v,v)  \; (\exp_x^\tau(\varepsilon \cdot))^*(\dif\bal^\tau_{x,\eps}).
\end{align*}
\par Due to the compact regime, we can assume the curvature is bounded all along the flow. With the assumption $\|\sec\|\le K$, the Jacobian of the exponential map satisfies
\begin{align}\label{eq:jac-exp}
1 - C_1nK\|v\|^2 \le \mathrm{Jac}(\exp_x)(v) \le 1 + C_2nK\|v\|^2.
\end{align}
\par In other words, one has the Radon-Nykodym estimates
\[
\frac{\dvol_\tau\restr_{B_x(\eps)}}{\mathrm{d} \left(\exp_x(\eps \cdot)\right)_\sharp \left(\mathcal{L}_{\R^n}\restr_{\mathbb{B}_0(1)}\right)} = 1 + O(\varepsilon^2),
\]
as well as
\[
\frac{\mathrm{d} \mathcal{L}_{\R^n}\restr_{\mathbb{B}_0(1)}}{\mathrm{d} \left(\exp_x(\eps \cdot)^* \dvol_\tau\restr_{B_x(\eps)}\right)} = 1 + O(\varepsilon^2),
\]
where $\vol_\tau$ is the Riemannian volume measure, $\mathcal{L}_{\R^n}$ is the Lebesgue measure and $\mathbb{B}_0(1) \subset \Tang_xM$ denotes the unit Euclidean ball in the tangent space.  
\par Since as $\varepsilon \downarrow 0$, we have
$
\nicefrac{\vol(B_\eps)}{\omega_n\eps^{n}} \to 1,
$
with the aid of \eqref{eq:jac-exp}, one observes the strong convergence of measures 
\[
(\exp_x^\tau(\varepsilon \cdot))^*(\bal^\tau_{x,\eps}) \to \bal_{0,1} = \frac{1}{\omega_n}\mathcal{L}_{\mathbb{T}_xM}, \quad \text{as $\varepsilon \downarrow 0$},
\]
where $ d\bal_{0,1}$ is the uniform probability measure on the unit Euclidean ball $\mathbb{B}_0(1) \subset \Tang_xM$. 
\par Putting these together, we conclude that \eqref{eq:limit-trace} converges to
\[
\int_{\|v\|_\tau \le 1} \partial_\tau g_{\tau}(v,v) \frac{d\mathcal{L}}{\omega_n} \le (=) \int_{\|v\|_\tau \le 1} h_\tau\left(v,v\right) \frac{d\mathcal{L}}{\omega_n} = \frac{1}{n+2} \tr_{g(\tau)}h_\tau.
\]
\par To prove the second claim, a similar argument shows 
\begin{align*}
	\lim_{\varepsilon \downarrow 0} \frac{1}{\varepsilon^{2}} \int_{\dist_{\tau}(x,y) = \varepsilon } \left[\partial_\tau \dist^2_{\tau}(x,y)\right]  \;\dif\s^\tau_{x,\eps} \le (=) \int_{\|v\|_\tau = 1} h_\tau\left(v,v\right) d\s_{0,1} = \frac{1}{n}\tr_{g(\tau)}h_\tau(x),
\end{align*}
where $\s_{0,1}$ is the uniform probability measure on the tangent unit sphere $\Sph_0(1) \subset \Tang_xM$.
\end{proof}
\begin{proof}[\footnotesize \textbf{Proof of Theorem~\ref{thm:Ric-Sat}}]
\hfill
\begin{enumerate}
		\item [] (1) $\Longrightarrow$ (2):  Suppose $g(\tau)$ is a (backward) Ricci flow i.e. $\partial_\tau g(\tau) = 2\Ric(g(\tau))$.
By the trace formula, we have
\begin{align*}
	\lim_{\varepsilon \downarrow 0} \frac{1}{\varepsilon^{2}} \int_{\dist_{\tau}(x,y) \le \varepsilon } \left[\partial_\tau \dist^2_{\tau}(x,y)\right]  \;\dif\bal^\tau_{x,\eps} = \;\; \frac{2}{n+2} \scal_{g(\tau)}(x).
\end{align*}
On the other hand
\begin{align}\label{eq:scal-asymp}
\frac{2}{n+2} \scal_{g(\tau)}(x) = 12 \lim_{\varepsilon \downarrow 0}  \left( 1 -  \frac{\mathcal{H}^{n} \left( B^{\tau}_{\varepsilon}(x) \right)}{\omega_n\varepsilon^n} \right).
\end{align}
\par Combining these two limits, gives the equality case in the saturation inequality.
\medskip
\item [] (2) $\Longrightarrow$ (1): Suppose $g(\tau)$ is super Ricci flow that is not a Ricci flow. Then there exists $\tau_0$ and $x_0$ such that the set
\[
\mathbb{S}_{x_0}M \supset \Omega_{\tau_0}  := \left\{ v \in \mathbb{S}_{x_0}M \quad \textrm{\textbrokenbar} \quad \partial_\tau g_{\tau_0}(v,v) < 2\Ric_{\tau_0}(v,v) \right\},
\]
has a positive $\mathcal{H}^{n-1}$ measure. In particular, tracing the proof of the trace formula, this implies
\begin{align*}
	\lim_{\varepsilon \downarrow 0} \frac{1}{\varepsilon^{2}} \int_{\dist_{\tau_0}(x_0,y) \le \varepsilon } \left[\partial_\tau \dist^2_{\tau}(x_0,y)\right]  \;\dif\bal^\tau_{x,\eps} < \frac{2}{n+2} \scal_{g(\tau_0)}(x_0), 
\end{align*}
which by virtue of \eqref{eq:scal-asymp}, is a contradiction to the saturation condition. 
\end{enumerate}
\end{proof}
\section{Epilogue: Ricci flow of metric-measure spaces}	\label{sec:wrf}
\par As we have seen in the previous section, in the smooth setting, on a smooth compact manifold, a Ricci flow is nothing but a saturated super Ricci flow. Both being a super Ricci flow and being saturated were expressed in terms of metric and measure data (i.e. weakly). 
\par Smooth super Ricci flows, as we have observed, can be weakly formulated in terms of coupled contraction of dynamic diffusions (in the \textsf{virtually psc} case) or by the monotonicity of the pointwise Lipschitz constant under the heat flow. Both characterizations have weak formulations i.e. formulations that solely use metric and measure data. 
\par Building on this observation, we conclude these notes by proposing the following notions of weak Ricci flows for a time-dependent metric-measure space.   
\par Recall the upper and lower derivatives for a function $f:\R \to \R$
\[
\frac{\dif^-}{\dif t}\restr_{t = t_0} f := \varliminf_{t \to t_0} \frac{f(t) - f(t_0)}{t - t_0}, \quad \frac{\dif^+}{\dif t}\restr_{t = t_0} f := \varlimsup_{t \to t_0} \frac{f(t) - f(t_0)}{t - t_0}.
\]	

\begin{definition}[Core of a time-dependent metric-measure space]
	Let $\overrightarrow{\mathcal{X}} = \left(X, \dist_\tau, \m_\tau \right)$ be time dependent metric-measure space. A set $\mathsf{X} \subset X$ is called a core if it is a full-measure set for every $\m_\tau$. 
\end{definition}

\begin{definition}[Weak Ricci flow]
\par Let $\overrightarrow{\mathcal{X}} = \left(X, \dist_t, \m_t  \right)$ be given of compact time-dependent pseudo metric-measure spaces ($\dist_t$ are pseudo metrics) of Hausdorff dimension $n$. 
\par $\overrightarrow{\mathcal{X}}$ is called a weak Ricci flow if it is a \textsf{WSRF} and if it is saturated in the sense that
\begin{align*}
	\frac{12(n+2)}{n}\left(\frac{\mathcal{H}^{n} \left( B^{t}_{\varepsilon}(x) \right)}{\omega_n\varepsilon^n} -1\right) \le \smallO(\eps^2) + \int_{\dist_{\tau}(x,y) \le \varepsilon } \left[\partial^+_t \dist^2_{t}(x,y)\right]  \;\dif\bal^t_{x,\eps}.
\end{align*}
as $\eps \downarrow 0$ holds for every $x$ in a core $\mathsf{X}$ of $\mathcal{X}$. 
\end{definition}
\begin{definition}[Weak $\cost$-Ricci flow]
\par Let $\overleftarrow{\mathcal{X}} = \left(X, \dist_\tau, \m_\tau, \cost_\tau  \right)$ be a given compact \textsf{virtually psc} flow of pseudo metric-measure spaces equipped with time-dependent cost functions. 
\par $\overleftarrow{\mathcal{X}}$ is said to be a weak $\cost$-Ricci flow if it is a $\cost$-\textsf{WSRF} that is saturated in the sense that
\begin{align*}
	  \int_{\dist_{\tau}(x,y) \le \varepsilon } \left[\partial^-_\tau \dist^2_{\tau}(x,y)\right]  \;\dif\bal^\tau_{x,\eps} \ge \frac{12(n+2)}{n}\left(1 - \frac{\mathcal{H}^{n} \left( B^{\tau}_{\varepsilon}(x) \right)}{\omega_n\varepsilon^n}\right) + \smallO(\eps^2)
\end{align*} 
as $\eps \downarrow 0$ holds for all $x$ in a core $\mathsf{X}$ of $\mathcal{X}$. 
\end{definition}
\renewcommand{\baselinestretch}{1.2}
\vspace{10pt}

\begingroup
\begin{thebibliography}{depth}
\bibitem{ACT}
	M. Arnaudon and K. A. Coulibaly, and A. Thalmaier,
	\textit{Horizontal diffusion in $\mathcal{C}^1$ path space}, S\'eminaire de Probabilit\'es XLIII, Lecture Notes in Math., {\bf vol. 2006}, Springer, Berlin, 2011, pp. 73--94.
	\bibitem{B}
	R.H. Bamler,  
	\textit{Compactness theory of the space of super Ricci flows}, Invent. math. {\bf 233} (2023), 1121--1277.
	\bibitem{Cha}
	I. Chavel, 
	\textit{Eigenvalues in Riemannian geometry}, Pure and Applied Mathematics, {\bf vol 115}, Academic Press, 1984.
	\bibitem{EN}
	K. J. Engel and R. Nagel,
	\textit{One-parameter semigroups for linear evolution equations}, Springer, New York, NY, 2000. 
	\bibitem{EK}
	S. N. Ethier and T. G. Kurtz,
	\textit{Markov processes: characterization and convergence}, New York: Wiley, 1986.
\bibitem{Faris}
	W.G. Faris, 
	\textit{Product formulas for perturbations of linear propagators},
J. Funct. Anal. {\bf 1} (1967) 93--108.
\bibitem{Gr}
M. Gromov,
\textit{Four lectures on scalar curvature
},(IHES notes), arXiv:1908.10612, (2019).
\bibitem{H1}
	R. S. Hamilton, 
	\textit{Three-manifolds with positive Ricci curvature},
	J. Differential Geom. {\bf 17} (1982), no. 2, 255--306.
\bibitem{H2}
	\bysame,
	\textit{Formation of singularities in the Ricci flow}, Surveys in Diff. Geom. {\bf 2} (1995), 7--136.
\bibitem{H3}
	\bysame,
	\textit{A compactness property for solutions of the Ricci
	flow},
	Amer. J. Math. {\bf 117} (1995), 545--572.
	\bibitem{HN}
	R. Haslhofer and A. Naber, 
	\textit{Characterizations of the Ricci  flow}, 
J. Eur. Math. Soc. (JEMS) {\bf 20} (2018), no. 5, 1269--1302.
\bibitem{KL}
	B. Kleiner and J. Lott,
\textit{Singular Ricci flows I},
	Acta Math. {\bf 219}, no. 1 (2017), 65--134.
\bibitem{KL2}
\bysame,
\textit{Singular Ricci Flows II},
Geometric Analysis. Progress in Mathematics, {\bf vol 333}. Birkh\"auser, Cham, 2020.
\bibitem{KS}
	E. Kopfer and K. T. Sturm, 
	\textit{Heat ﬂow on time-dependent metric measure spaces and super-Ricci flows}, Comm. Pure Appl. Math. {\bf 71}, no. 12 (2018), 2500--2608.
\bibitem{LM1}
	S. Lakzian and M. Munn, 
	\textit{Metric perspectives of the Ricci flow applied to disjoint unions}, 
Anal. Geom. Metr. Spaces {\bf 2} (2014), 282--293.
\bibitem{LM2}
	\bysame, 
	\textit{On weak super Ricci flow through neckpinch},
	 Anal. Geom. Metr. Spaces {\bf 9} (2021), 120--159.
	\bibitem{Li}
	X.-D. Li, 
\textit{On Perelman’s W-entropy and Shannon entropy power for super Ricci flows on metric measure spaces}, https://arxiv.org/abs/2505.03202
\bibitem{MT}
	R. J. McCann and P. M. Topping,
	\textit{Ricci flow, entropy and optimal transportation}, 
Amer. J. Math. {\bf 132} no. 3 (2010), 711--730.
	\bibitem{Paz}
	A. Pazy,
	\textit{Semigroups of linear operators and applications to partial differential equations}, Springer, New York, NY, 1983. 
\bibitem{G1}
	G. Perelman, 
	\textit{The entropy formula for the Ricci flow and its geometric applications},
 arXiv:math/0211159 (2002).
\bibitem{G2}
\bysame,
	\textit{Finite extinction time for the solutions to the Ricci flow on certain three-manifolds}, arXiv:math/0307245 (2003).
	\bibitem{G3}
\bysame,
	\textit{Ricci flow with surgery on three-manifolds}, arXiv:math/0303109 (2003).
\bibitem{VS}
M. K. von Renesse and K. T. Sturm,
\textit{Transport inequalities, gradient estimates, entropy, and Ricci curvature}, 
Comm. Pure Appl. Math. {\bf 58} no. 7 (2005), 923--940.
	\bibitem{St1}
	K. T. Sturm,
	\textit{Super-Ricci flows for metric measure spaces}, 
J. Funct. Anal.  {\bf 275} no. 12 (2018), 3504--3569.
	\bibitem{Top2}
	P. Topping,
	\textit{Ricci flow: the foundations via optimal transportation},  Optimal Transportation, Theory and Applications, LMS lecture notes series, {\bf vol. 413}, CUP, 2014.
	\bibitem{V}
	P.-A. Vuillermot,
	\textit{A generalization of Chernoff’s product formula
		for time-dependent operators},
J. Funct. Anal. {\bf 259} (2010), 2923--2938.	
	


	 
\end{thebibliography}
\endgroup
\end{document}